\documentclass[12pt]{amsart}
\usepackage{amsmath,amssymb, xy,wasysym,bm,amsthm}
\usepackage{tikz}
\usepackage{graphicx}
\usepackage{hyperref}
\usepackage{mathrsfs}
\usepackage[capitalize,nameinlink,noabbrev,nosort]{cleveref}
\theoremstyle{plain}
\newtheorem{theorem}{Theorem}[section]
\newtheorem{lemma}[theorem]{Lemma}
\newtheorem{proposition}[theorem]{Proposition}

\theoremstyle{definition}
\newtheorem{definition}{Definition}[section]

\newtheorem{example}{Example}[section]

\newcommand{\mr}{\mathsf{mr}}

\title{}
\begin{document}
\title[An exact enumeration of vertex connectivity]{An exact enumeration of vertex connectivity of the enhanced power graphs of finite nilpotent groups}
\author[Sudip Bera]{Sudip Bera}
\author[Hiranya Kishore Dey]{Hiranya Kishore Dey}
\address[Sudip Bera]{Faculty of Mathematics, DA-IICT, Gandhinagar 382007}
\email{sudip\_bera@daiict.ac.in}
\address[Hiranya Kishore Dey]{Department of Mathematics, Indian  Institute of Science, Bangalore 560 012}
\email{hiranyadey@iisc.ac.in} 

\maketitle	
\begin{abstract}
The enhanced power graph of a group $G$ is a graph with vertex set $G,$ where two distinct vertices $x$ and $y$ are adjacent if and only if there exists an element $w$ in $G$ such that both $x$ and $y$ are powers of $w.$ In this paper, we determine the vertex connectivity of the enhanced power graph of any finite nilpotent group.
\end{abstract}
\section{Introduction}
\label{sec:intro} 
The examination of graphs arising from groups provides several benefits. Firstly, it permits us to classify the resulting graphs. Secondly, it empowers us to identify algebraic structures that possess isomorphic graphs. Lastly, it aids in comprehending the interdependence between algebraic structures and their corresponding graphs. Moreover, these graphs have numerous valuable applications, as demonstrated by sources such as \cite{surveypwrgraphkac1, cayleygraphsckry}. Furthermore, they are closely connected to automata theory \cite{automatatheory}. Various types of graphs exist, including but not limited to commuting graphs of groups \cite{ijraeljofmathematics, braurflower, onboundingdiamcommuting}, power graphs of semigroups \cite{undpwrgraphofsemgmainsgc1, directedgrphcompropofsemgrpkq3}, groups \cite{combinatorialpropertyandpowergraphsofgroupskq1}, intersection power graphs of groups \cite{intersectionpwegraphb3}, enhanced power graphs of groups \cite{firstenhcedpwrstrctreaacns1,enhancedpwrgrapbb3,enhancedpower-jkum}, and comaximal subgroup graphs \cite{das-saha-saba}. These graphs have been established to investigate the properties of algebraic structures using graph theory.

The commuting graph, a highly significant and extensively researched graph, is linked to a group $G$. This graph has been thoroughly investigated in \cite{ijraeljofmathematics, braurflower, onboundingdiamcommuting}. To formally define the commuting graph, we turn to the following:

\begin{definition} [\cite{ijraeljofmathematics, braurflower, onboundingdiamcommuting}]
Consider a group $G$. The commuting graph of $G$, denoted as $\mathcal{C}(G)$, is a simple graph where the vertex set comprises non-central elements of $G$. Two distinct vertices $u$ and $v$ are connected in this graph if and only if $u$ and $v$ commute, that is, $uv=vu$.
\end{definition}

Around 2000, Kelarev and Quinn \cite{combinatorialpropertyandpowergraphsofgroupskq1} introduced the concept of a power graph in the realm of semigroup theory.
\begin{definition}[\cite{surveypwrgraphkac1,undpwrgraphofsemgmainsgc1, combinatorialpropertyandpowergraphsofgroupskq1}]\label{defn: powr graph}
Given a group $G,$ the \emph{power graph} $\mathcal{P}(G)$ of $G$ is a simple graph with vertex set $G$, where two vertices $u$ and $v$ are connected by an edge if and only if one of them is the power of the other.
\end{definition}
In this paper, our topic of interest is a novel graph, called the \emph{enhanced power graph}, which was introduced by Aalipour et al. \cite{firstenhcedpwrstrctreaacns1} to evaluate the proximity between the power graph and the commuting graph of a group $G.$ 
\begin{definition}[\cite{firstenhcedpwrstrctreaacns1}]\label{defn: enhcdpowr graph}
Let $G$ be a group. The \emph{enhanced power graph} of $G,$ denoted as $\mathcal{G}_E(G),$ is the graph with the vertex set $G,$ where two vertices $u$ and $v$ are connected if and only if there exists an element $w \in G$ such that both $u \in \langle w \rangle $ and $v \in \langle w \rangle.$

Given a group $G,$ the \emph{proper enhanced power graph} of $G,$ denoted by $\mathcal{G}^{**}_E(G),$ is the graph obtained by deleting all the dominating vertices from the enhanced power graph $\mathcal{G}_E(G).$ Moreover, by $\mathcal{G}^{*}_E(G)$ we denote the graph obtained by deleting only the identity element of $G$ and this is called \emph{deleted enhanced power graph} of $G.$ Note that if there is no such dominating vertex other than identity, then $\mathcal{G}_{E}^{*}(G)=\mathcal{G}_{E}^{**}(G).$
\end{definition}
Many works have been done in the last decade that investigate various properties of the enhanced power graphs of finite groups. Aalipour et al. \cite{firstenhcedpwrstrctreaacns1} 
characterized the finite groups such that any arbitrary pair of these three graphs (power, commuting, enhanced) is equal. 
Ma and She  in \cite{ma-she} derived the metric dimension whereas Hamzeh et al. in \cite{Hamzeh-ashrafi} derived the automorphism groups of enhanced power graphs of finite groups. Determining the vertex connectivity of the power graph and the enhanced power graph of a group has been a challenging and well-investigated problem for the last few years. Many researchers have attempted and found several results for the vertex connectivity of the power graphs and enhanced power graphs. Chattopadhyay, Patra, and Sahoo in \cite{chatto-sahoo-patro, chatto-sahoo-patro-ii} found the exact vertex connectivity for the power graph of the cyclic group $\mathbb{Z}_n$ for most of the values of $n$. In \cite{chatto-sahoo-patro-iii} they have given exact values for the vertex connectivity of power graphs associated with the nilpotent groups which have all Sylow $p$-subgroups cyclic, except possibly one. Bera et. al. in \cite{bera-dey-sajal} gave a bound for the vertex connectivity of the enhanced power graph of any abelian group and in a subsequent work \cite{bera-dey-jgt} Bera and Dey improved the bound. 
In this paper, our focus is to determine the exact value of the vertex connectivity of the enhanced power graphs of finite nilpotent groups. 

\subsection{Basic definitions and notations}\label{subsection:definitions}
For the convenience of the reader and also for later use, we recall some basic definitions and notations about graphs. 
Let $\Gamma$ be a graph with vertex set $V$. 
Two elements $u$ and $v$ are said to be adjacent if there is an edge between them. 
A \emph{path} of length $k$ between two vertices $v_0$ and $v_k$ is an alternating
sequence of vertices and edges $v_0, e_0, v_1, e_1, v_2, \cdots , v_{k-1}, e_{k-1}, v_k$, where the $v_i'$s are distinct
(except possibly the first and last vertices) and $e_i'$s are the edges $(v_i, v_{i+1}).$  A graph $\Gamma$ is said to be \emph{connected} if for any pair of vertices $u$ and $v,$ there exists a path between $u$ and $v.$ 
 $\Gamma$ is said to be \emph{complete} if any two distinct vertices are adjacent. 
A vertex of a graph $\Gamma$ is called a \emph{dominating vertex} if it is adjacent to every other
vertex. 
The \emph{vertex connectivity} of a graph $\Gamma,$ denoted by $\kappa{(\Gamma)}$, is
the minimum number of vertices that need to be removed from the vertex set $\Gamma$ so that the
induced subgraph of $\Gamma$ on the remaining vertices is disconnected. The complete graph with $n$ vertices has vertex connectivity $n-1.$ 

Throughout this paper we consider $G$ as a finite group.  $|G|$ denotes the cardinality of the set $G.$ For a prime $p,$ a group $G$ is said to be a $p$-group if $|G|=p^{r}, r\in \mathbb{N}.$ If $|G|= p_{\ell}^{r}$ for some prime $p_{\ell}$, then we say that $G$ is a $p_{\ell}$-group.
 For any element $g \in G, \text{o}(g)$ denotes the order of the element $g.$ Let $m$ and $n$ be any two positive integers, then the greatest common divisor of $m$ and $n$ is denoted by $\gcd(m, n).$ The Euler's phi function $\phi(n)$ is the number of integers $k$ in the range $1 \leq k \leq n$ for which the  $\gcd(n, k)$ is equal to $1.$
   
\section{Main Results}
In this section, we will showcase our primary findings. To begin, we will first review the structure of a finite nilpotent group.
\subsection{Nilpotent group}\label{Defn: Nilpotent group}
A finite group $G$ is nilpotent if and only if $G\cong P_1\times \cdots\times P_r,$ where for each $i\in[r], P_i$ is a Sylow subgroup of order $p_i^{t_i}$ of $G.$ Let $G_1$ be a finite nilpotent group having no Sylow subgroups which are either cyclic or generalized quaternion. Now, for a finite nontrivial nilpotent group $G,$ we have the following cases:
\begin{enumerate}
\item 
No Sylow subgroups of $G$ are either cyclic or generalized quaternion.
\item
$G$ has cyclic Sylow subgroups. In this case, $G = G_1\times \mathbb{Z}_n,$ where $G_1$ is described as above and  $\gcd(|G_1|, n)=1.$
\item
$G$ has a Sylow subgroup isomorphic to generalized quaternion. Here $G = G_1\times Q_{2^k},$  and $\gcd(|G_1|, 2)=1.$
\item
$G$ has both a cyclic Sylow subgroup and a Sylow subgroup isomorphic to generalized
quaternion. In this case, $G = G_1\times \mathbb{Z}_n\times Q_{2^k},$ where  $\gcd(|G_1|, n)=\gcd(|G_1|, 2)=\gcd(n, 2)=1.$ 
\end{enumerate}

In \cite{enhancedpower-jkum}, Kumar et. al. have obtained the following result for the vertex connectivity of the enhanced power graphs of certain nilpotent groups. 

\begin{theorem}\cite[Theorem 5.6]{enhancedpower-jkum}
\label{kumar-ma-parveen}
Let $G=P_1 \times  \cdots \times P_r $ be a non-cyclic nilpotent group with $r \geq 2$. Suppose that each Sylow subgroup of $G$ is cyclic except $P_k$ for some $k \in [r].$
\begin{enumerate}
    \item If $P_k$ is not generalized quaternion, then $\kappa(\mathcal{G}_E(G))= \prod_{i=1, i\neq k} ^r |P_i|. $

    \item If $P_k$ is generalized quaternion, then $\kappa(\mathcal{G}_E(G)= 2 \prod_{i=1, i\neq k} ^r |P_i| $
\end{enumerate}

\end{theorem}
As we can see, their result is for those nilpotent groups, which have all Sylow $p$-subgroups cyclic, except possibly one. That is, they have given the exact vertex connectivity for only those nilpotent groups which have the form $P\times \mathbb{Z}_n $ where $P$ is a non-trivial $p$-group and $\gcd(n,p)=1$. In this paper, we give the vertex connectivity of the enhanced power graph of any nilpotent group, in terms of the number of minimum {\it roots}, which we define next.

For any group $G$ and an element $a \in G$, let $\mathrm{Roots}(a)$ denote the set of elements $g \in G$ such that $a \in \langle g \rangle.$  
For a $p$-group $P$, define 
\begin{equation}\label{Equn: defin of root}
 \mr(P)=\min_ { a \in P \mbox{ and } o(a)=p} |\mathrm{Roots} (a)|.     
\end{equation}
For example, the group $P=\mathbb{Z}_3 \times \mathbb{Z}_9$ has 
$\mr(P)=2$ and  the dihedral group in $8$ elements, which we denote by $D_8$, has $\mr(D_8)=1$.

Now for the group $G_1=P_1\times \cdots\times P_r,$ where each $P_i$ is neither cyclic nor generalized quaternion, we define 
\begin{equation}\label{equa: tau for nilpo grp neither cyclic nor Q}
\tau=\min_{T \subseteq [r]} \left\{ \prod_{i \in T} \big| P_i \big| \left(\prod _{i \in [r]\setminus T} (\mr(P_i)+1) - \prod_{i \in [r]\setminus T}\mr(P_i) \right) \right\}. 
\end{equation}
When $r=1$, it is clear that $\tau=1$. 
With the above definition, we can now state our main result for the vertex connectivity of the enhanced power graph of a finite nilpotent group which does not have any Sylow subgroups that are generalized quaternion. 
\begin{theorem}\label{Thm:VC Enh Nil with Cyclic only}
Let $G$ be a non-cyclic finite nilpotent group that does not have any Sylow subgroups which are generalized quaternions. Then 
\begin{equation*}
\kappa(\mathcal{G}_E(G))=
\begin{cases}
\tau,  & \text{ if } G=G_1, \text{ where } G_1 \text{ is described as above}, \\  
 n \tau, & \text{ if } G=G_1\times \mathbb{Z}_n, \text{ and } \gcd(|G_1|, n)=1,
\end{cases}  
\end{equation*}
where $\tau$ is taken from \eqref{equa: tau for nilpo grp neither cyclic nor Q}.
\end{theorem} 
Now we consider the finite nilpotent group $G$ such that $G$ has a generalized quaternion Sylow subgroup. In this case, we can consider $G=G_1\times Q_{2^k}=P_1\times \cdots\times  P_r\times Q_{2^k},$ where $\gcd(|P_i|,2)=1.$ For the next definition, we write $Q_{2^k}$ as $P_{r+1}$. Define
\begin{equation}\label{Equa: mu for nilpo grp having Q}
\mu=\min_{T \subseteq [r+1]} \left\{ \prod_{i \in T} \big| P_i \big| \left(\prod _{i \in [r+1]\setminus T} (\mr(P_i)+x) - \prod_{i \in [r+1]\setminus T}\mr(P_i) \right) \right\},
\end{equation}
\text{ where } $x=1$ \text{ if } $i \neq r+1$ \text{ and } $x=2, \mr(P_i)=2$ \text{ if } $i=r+1.$
\begin{theorem}
\label{thm:main-nilp-with-cyc-gen-quat}
Let $G$ be a finite nilpotent group having a Sylow subgroup which is generalized quaternion. Then 
\begin{equation*}
\kappa(\mathcal{G}_E(G))=
\begin{cases}
\mu,  & \text{ if } G=G_1\times Q_{2^k} \text{ where } G_1 \text{ is described as above and } \gcd(|G_1|, 2)=1 \\  
 n \mu, & \text{ if } G=G_1\times \mathbb{Z}_n\times Q_{2^k} \text{ and } \gcd(|G_1|, n)=\gcd(|G_1|, 2)=\gcd(n, 2)=1,
\end{cases}  
\end{equation*}
where $\mu$ is taken from \eqref{Equa: mu for nilpo grp having Q}.
\end{theorem}

From \cref{thm:main-nilp-with-cyc-gen-quat}, it is clear that when $G= \mathbb{Z}_n \times Q_{2^k}$, $\kappa(\mathcal{G}_E(G))=2n$. Moreover, when $G= P_1 \times \mathbb{Z}_n $, we 
set $G_1=P_1$ in \cref{Thm:VC Enh Nil with Cyclic only} to obtain  $\kappa(\mathcal{G}_E(G))=n$.
Thus, \cref{kumar-ma-parveen} is an immediate corollary of \cref{Thm:VC Enh Nil with Cyclic only} and \cref{thm:main-nilp-with-cyc-gen-quat}.



Bera and Dey in \cite{bera-dey-jgt} proved the following upper bound for the enhanced power graph of any finite abelian group.

\begin{theorem}\cite[Theorem 4.8]{bera-dey-jgt} 
\label{improved vc of enhced pwr raph for grnrral abelian grp}	
	Let 
	\begin{equation*} 
 \label{eqn:G_form} 
	G=\mathbb{Z}_{p^{t_{11}}_1}\times\cdots\times\mathbb{Z}_{p^{t_{1k_1}}_1}\times\cdots\times
\mathbb{Z}_{p^{t_{r1}}_r}\times\cdots\times\mathbb{Z}_{p^{t_{rk_r}}_r}\times \mathbb{Z}_n 
	\end{equation*} 
	where $k_i \geq 2$, $\gcd(n, p_i)=1$, and $1\leq t_{i1}\leq t_{i2}\leq\cdots\leq t_{ik_i} $, for all $i\in [r]. $
	We then have, 
\[ \kappa(\mathcal{G}_E(G))\leq n \left(p_1^{t_{11}}p_2^{t_{21}}\cdots p_r^{t_{r1}}-\phi(p_1^{t_{11}}p_2^{t_{21}}\cdots p_r^{t_{r1}})\right).\] 
\end{theorem} 

In the following result,  we explicitly determine the vertex connectivity of the enhanced power graph for any finite abelian group. In particular, we prove the following: 

\begin{theorem}\label{thm:vc-abelian-exact}
Let 
\begin{equation*}
G=\mathbb{Z}_{p^{t_{11}}_1}\times\cdots\times\mathbb{Z}_{p^{t_{1k_1}}_1}\times \cdots\times
\mathbb{Z}_{p^{t_{r1}}_r}\times\cdots\times\mathbb{Z}_{p^{t_{rk_r}}_r} \times \mathbb{Z}_n
\end{equation*}	
where $k_i\geq 2 , \gcd(n,p_i)=1$, and $1\leq t_{i1}\leq t_{i2}\leq\cdots\leq t_{ik_i} $, for all $i\in [r].$
Then, 
\[ \kappa( \mathcal{G}_E(G)= n \min_{T \subseteq [r]} \left\{ \prod_{i \in T} \left| p_i ^ { \sum_{\ell=1}^{k_i}t_{i\ell} } \right| 
\left(\prod _{i \in [r]\setminus T} 
\left( \frac{p_i^{t_{i1}k_i}-1}{p_i^{k_i}-1} +1 \right) - \prod_{i\in [r] \setminus T} 
\left ( \frac{p_i^{t_{i1}k_i} - 1} {p_i^{k_i}-1}
\right)
\right) \right\}.\] 
\end{theorem}

\section{Preliminaries} 
In this section, we will start by reviewing specific findings that have been previously established and are pertinent to our paper. Bera et al. carried out an extensive examination on the fascinating properties of enhanced power graphs of finite groups. Their studies in references \cite{enhancedpwrgrapbb3} and \cite{bera-dey-sajal} successfully demonstrated the outcomes as follows:
\begin{lemma}
\cite[Theorem 2.4]{enhancedpwrgrapbb3}
\label{lem:complete}
The enhanced power graph $\mathcal{G}_E(G)$ of the group $G$ is complete if and only if $G$ is cyclic.
\end{lemma}
\begin{lemma}
\label{lem;enhanced-coprime} 
\cite[Lemma 2.5]{bera-dey-sajal}
 Let $G$ be a finite group and let $x, y \in G$ be
such that $\gcd\{o(x), o(y)\}= 1$ and $xy = yx$. Then $x \sim y$ in $\mathcal{G}_E(G)$.
\end{lemma}
\begin{lemma}\cite[Lemma 2.6]{bera-dey-sajal}
\label{lema: p and p^i orderd path connd abelong< b>}
Let $G$ be a $p$-group. Let $a, b$ be two elements of $G$ of order $p, p^i (i\geq 1)$ respectively. If there is a path between $a$ and $b$ in $\mathcal{G}_E^*(G),$ then $\langle a\rangle \subseteq \langle b\rangle.$ In particular, if both a and b have order p, then, $\langle a \rangle =\langle b \rangle.$
\end{lemma}

The following lemma follows from \cref{lema: p and p^i orderd path connd abelong< b>}

\begin{lemma}
\label{lem:p-grp-roots}
Let $G$ be a $p$-group which is neither cyclic nor generalized quaternion. Then, the minimum size of a connected component in $\mathcal{G}_E^{*}(G)$ is $\mr(G)$.
\end{lemma}


We recall the definition of the generalized quaternion group. Let $x = \overline{(1, 0)}$ and $y = \overline{(0, 1)}.$ Then $Q_{2^n} = \langle x, y\rangle,$ where
\begin{enumerate}
\item
$x$ has order $2^{n-1}$ and $y$ has order $4,$
\item
every element of $Q_{2^n}$ can be written in the form $x^a$ or $x^ay$ for some $a\in \mathbb{Z},$
\item
$x^{2^{n-2}}=y^2,$
\item
for each $g\in Q^{2^n}$ such that $g\in \langle x \rangle,$ such that $gxg^{-1}=x^{-1}.$
\end{enumerate}
For more information about $Q_{2^n}$ see \cite{generalized-quaternion, algebradummitfoote, scott-group}. 
\begin{lemma}[Theorem 4.2, \cite{bera-dey-sajal}]\label{dom of gen Q_2^n in enhced}
The enhanced power graph of generalized quaternion group $Q_{2^n}$ has vertex connectivity $2$. Moreover, any separating set of $\mathcal{G}_E(Q_{2^n})$ must contain the identity element and also the unique non-identity element of order $2$. 
Moreover, it is easy to show that the minimum size of a connected component in $\mathcal{G}_E^{**}(G)$ is $2$. 
\end{lemma}

\begin{lemma}
\label{lem;enhanced-pgrp} 
 \cite[Theorem 1.1]{bera-dey-sajal}
 Let $G$ be a finite $p$-group that is neither cyclic nor
 generalized quaternion. Then $\kappa(\mathcal{G}_E(G))=1.$
 \end{lemma}

\subsection{Strong product of connected graphs} 
We will now review the definition and some fundamental results related to the strong product of $r$ connected graphs. Suppose $G_i= (V_i, E_i)$ are simple graphs, where $i=1,2,\cdots,r$. The strong product $G_1 \boxtimes  \cdots \boxtimes G_r$ of graphs $G_1, \cdots, G_r$ is a graph such that the vertex set of $G_1 \boxtimes  \cdots \boxtimes G_r$ is 
\[V(G_1 \boxtimes \cdots \boxtimes G_r)= V_1 \times \cdots \times V_r\]  and two distinct vertices $(x_1, \cdots,x_r)$ and $(y_1,\cdots,y_r)$ are adjacent in $G_1 \boxtimes \cdots \boxtimes G_r$ if and only if $x_i=y_i$ or  $x_iy_i \in E_i$ for $i=1, 2, \cdots, r.$

Spacapan in \cite{spacapan} determined the vertex connectivity for the strong product of $k$ connected graphs. We first discuss the vertex connectivity of strong product of $2$ graphs in details so that it becomes clear to the reader. 
Let $G_1= (V_1, E_1)$ and $G_2=(V_2, E_2)$ be two connected graphs. Let $S_1$ be a separating set in $G_1$ and $S_2$ be a separating set in $G_2$. 
Clearly $S_1 \times V_2$ and $V_1 \times S_2$ are separating sets in $G_1 \boxtimes G_2$. These kinds of separating sets are called {\it I-sets}. 
Moreover, if $A_1, A_2, \dots, A_k$ are connected components of $G_1 - S_1$ and $B_1, B_2, \dots, B_{\ell}$ are connected components of $G_2- S_2$, then for any $1 \leq i \leq \ell$ and $ 1 \leq j \leq k$,
\[ (S_1 \times B_i) \cup (S_1 \times S_2) \cup (A_j \times S_2)\]
is a separating set in $G_1 \boxtimes G_2.$ Also, observe that 
\[ (S_1 \times B_i) \cup (S_1 \times S_2) \cup (A_j \times S_2)= [ (A_j \cup S_1) \times (B_i \times S_2) ] \setminus (A_j \times B_i).\] These kind of separating sets are called {\it L-sets}. 


We now formally define the I-sets and L-sets for the strong product of $r$ connected graphs.  Suppose $G_i= (V_i, E_i)$ are graphs, where $i=1,2,\cdots,r$. 
Let $F \subseteq [r]$ be a nonempty subset and for each $i \in F$, let $S_i$ be a separating set in $G_i$. Then, $G_i \setminus S_i$ is disconnected. Let $A_i$ be one of the connected components of $G_i-S_i$.
Further, if $i \in F$, define $U_i= S_i \cup A_i$, $W_i=A_i$ and if $ i \notin F$, define $U_i=W_i=V_i$. Let 
\begin{equation}
\label{eqn:spac} 
S_F= \prod_{i=1}^n U_i- \prod _{i=1}^n W_i=\left( \prod _{i \in F} S_i \cup A_i- \prod_{i \in F} A_i \right)\prod _{i \in [r]\setminus F} V_i.
\end{equation}
It is not difficult to see that any set $S_F$, constructed this way, is a separating set for the graph $G_1\boxtimes \cdots \boxtimes G_r.$ We call
$S_F$ an I-set if $|F|=1$ and an L-set if $|F|>1$. When $r=2$, it is clear that this definition matches with the definition made above. 

Spacapan \cite[Theorem 4.1, Corollary 4.2]{spacapan} proved the following.
\begin{theorem}[Spacapan]
\label{thm:spacapan-strong-prod}
Let $r \geq 2$ and $G_i=(V_i, E_i)$ be connected graphs for $i=1, 2,\cdots, r$. Then, any separating set in $G_1 \boxtimes  \cdots \boxtimes G_r$ is an $I$-set. Therefore, the vertex connectivity of  $G_1 \boxtimes G_2 \boxtimes \dots \boxtimes G_r$ is 
equal to the minimum size of an I-set or L-set.
\end{theorem}

The following proposition follows from \cref{lem;enhanced-coprime} and the definition of the enhanced power graph. 
\begin{proposition}
\label{prop:enhanced-power-strong-prod}
Let $G = P_1\times \cdots \times P_r$ where each $P_i$ is a $p$-group with respect to the prime $p_i$.
Then $\mathcal{G}_E(G)= \mathcal{G}_E(P_1) \boxtimes \cdots \boxtimes \mathcal{G}_E(P_r).$
\end{proposition}

We next prove the following inequality concerning two sequences of positive integers.
\begin{lemma}
\label{lem:ineq1}
Let $a_1,\cdots, a_r, b_1, \cdots ,b_r$ be positive integers such that $a_i \geq b_i$ for $1 \leq i \leq r$. Then, 
\[ \prod_{i=1}^r(a_i+1)- \prod_{i=1}^ra_i \geq \prod_{i=1}^r(b_i+1)- \prod_{i=1}^rb_i.\]
\end{lemma}
\begin{proof}
We observe that
\begin{equation*}
  \prod_{i=1}^r(a_i+1) - \prod_{i=1}^r(b_i+1)  =   \prod_{i=1}^{r-1}(a_i+1)a_r - \prod_{i=1}^{r-1}(b_i+1)b_r +  \prod_{i=1}^{r-1}(a_i+1) -    \prod_{i=1}^{r-1}(b_i+1) 
\end{equation*}
As $a_r \geq b_r$, we have 
\begin{equation*}
  \prod_{i=1}^r(a_i+1) -    \prod_{i=1}^r(b_i+1)  \geq   \prod_{i=1}^{r-1}(a_i+1)a_r - \prod_{i=1}^{r-1}(b_i+1)b_r  
\end{equation*}
Using the same idea repeatedly, we can get 
\begin{equation*}
  \prod_{i=1}^r(a_i+1) -    \prod_{i=1}^r(b_i+1)  \geq   \prod_{i=1}^{r}a_i - \prod_{i=1}^{r}b_i.
\end{equation*}
This completes the proof. 
\end{proof}  

\section{Proofs of main theorems}

In this section, we prove our main results. Towards that, we start with the following proposition. 

\begin{proposition}\label{Prop: VC enh Nilpo no cyclic no Q}
Let $G_1=P_1 \times \cdots \times P_r$ be a finite nilpotent group having no Sylow subgroups which are either cyclic or generalized quaternion. Then 
$$\kappa(\mathcal{G}_E(G_1))= \tau=\min_{T \subseteq [r]} \left\{ \prod_{i \in T} \big| P_i \big| \big(\prod _{i \in [r]\setminus T} (\mr(P_i)+1) - \prod_{i \in [r]\setminus T}\mr(P_i) \big )\right\}.$$
\end{proposition}

\begin{proof}
Initially, we show that we can construct a separating set of the above size. For any $T \subseteq [r]$,  it is enough to construct a separating set of size 
\begin{equation} \label{eqn:req-size}
\prod_{i \in T} \big| P_i \big| \big(\prod _{i \in [r]\setminus T} (\mr(P_i)+1) - \prod_{i \in [r]\setminus T}\mr(P_i) \big ).
\end{equation}
By \cref{lem;enhanced-pgrp}, for $1 \leq i \leq r$, $\mathcal{G}_E(P_i)$ has vertex connectivity $1$ and $\{e_i\}$ is a separating set where $e_i$ denotes the identity element of the subgroup $P_i$. Let $S_i= \{e_i\}$ for $i \notin T$. By \cref{lem:p-grp-roots}, the minimum size of a connected component of $\mathcal{G}_E^{*}(P_i)$ is clearly $\mr(P_i).$ Here we take $A_i$ to be one of the minimum size connected components of 
$\mathcal{G}_E^{*}(P_i)$. Then $|A_i|=\mr(P_i)$ for $i \notin T$. Now, we can construct a separating set $S_{T^c}$ from $T$ by using \eqref{eqn:spac}, where $S_{T^c}$ is the following:
\begin{equation}\label{Eqn: a separating set S_T}
 \left( \prod _{i  \in  [r] \setminus T} A_i \cup \{e\}- \prod_{i \in [r] \setminus T} A_i \right)\prod _{i \in  T} P_i.   
\end{equation}
Moreover, the size of $S_{T^c}$ is \eqref{eqn:req-size}. 
		
We next need to show that any separating set of $\mathcal{G}_E(G_1)$ must be bigger than the above quantity. 		
By \cref{prop:enhanced-power-strong-prod} and  \cref{thm:spacapan-strong-prod}, any separating set must be of the form 
\[ \left( \prod _{i \in F} S_i \cup B_i- \prod_{i \in F} B_i \right)\prod _{i \in [r]\setminus F} P_i \]
where $F$ is a subset of $[r]$, and for each $i \in F$, $S_i$ is a separating set of $\mathcal{G}_E(P_i)$ and $B_i$ is some connected component of $\mathcal{G}_E^{*}(P_i)$. 
Now, any separating set $S_i$ of  $\mathcal{G}_E(P_i)$ has size at least $1$. Therefore, we have 
\begin{equation}
\label{eqn:imp4}
  \left( \prod _{i \in F} | S_i \cup B_i |- \prod_{i \in F} | B_i | \right)\prod _{i \in [r]\setminus F} | P_i | \geq 
   \left( \prod _{i \in F}( |  B_i |+1)- \prod_{i \in F} | B_i | \right)\prod _{i \in [r]\setminus F} | P_i |. 
   \end{equation}
 As any connected component of 
$\mathcal{G}_E^{*}(P_i)$ has size at least $\mr(P_i)$, we have 
$|B_i| \geq |\mr(P_i)|$. Therefore, using \cref{lem:ineq1}, we have
   \begin{equation}
   \label{eqn:imp5}
  \left( \prod _{i \in F}( |  B_i |+1)- \prod_{i \in F} | B_i | \right)  \prod _{i \in [r]\setminus F} | P_i |   \geq  \left( \prod _{i \in F}( |  \mr(P_i) |+1)- \prod_{i \in F} | \mr(P_i) | \right)  \prod _{i \in [r]\setminus F}  | P_i | .\end{equation}
Combining \eqref{eqn:imp4} and \eqref{eqn:imp5}, we have 
\begin{equation}
  \left( \prod _{i \in F} | S_i \cup B_i |- \prod_{i \in F} | B_i | \right)\prod _{i \in [r]\setminus F} | P_i | \geq  
  \left( \prod _{i \in F}( |  \mr(P_i) |+1)- \prod_{i \in F} | \mr(P_i) | \right)\prod _{i \in [r]\setminus F} | P_i | .
\end{equation}
 Thus, we get
\[ \left( \prod _{i \in F} | S_i \cup B_i |- \prod_{i \in F} | B_i | \right)\prod _{i \in [r]\setminus F} | P_i | \geq \min_{F \subseteq [r]} \left\{ \prod_{i \in [r]\setminus F} \big| P_i \big| \big(\prod _{i \in F} (\mr(P_i)+1) - \prod_{i \in F}\mr(P_i) \big )\right\}. \] This completes the proof. 
\end{proof}
\begin{proof}[Proof of \cref{Thm:VC Enh Nil with Cyclic only}]
First, suppose that $G_1= P_1 \times \cdots \times P_r$. In this case, $\kappa(\mathcal{G}_E(G_1))=\tau$ by \cref{Prop: VC enh Nilpo no cyclic no Q}. Now suppose that $G=G_1\times \mathbb{Z}_n,$ where $\gcd(|G_1|, n)=1.$ Then, by
\cref{prop:enhanced-power-strong-prod}
\begin{equation*}
\mathcal{G}_E(G)= \mathcal{G}_E(P_1) \boxtimes \cdots \boxtimes \mathcal{G}_E(P_r)\boxtimes \mathcal{G}_E(\mathbb{Z}_n).
\end{equation*}
Using \cref{lem:complete}, we have 
\begin{equation*}
\mathcal{G}_E(G)= K_n \boxtimes \mathcal{G}_E(P_1) \boxtimes \cdots \boxtimes \mathcal{G}_E(P_r).
\end{equation*}
where $K_n$ denotes the complete graph on $n$ vertices. Therefore, we clearly have 
\begin{equation*}
\kappa(\mathcal{G}_E(G))= n \kappa(\mathcal{G}_E(G_1)). 
\end{equation*}
The proof is now complete by using \cref{Prop: VC enh Nilpo no cyclic no Q}. 
\end{proof}

\begin{proof}[Proof of \cref{thm:main-nilp-with-cyc-gen-quat}]
 We at first consider the case when $G= G_1 \times Q_{2^k}.$
By
\cref{prop:enhanced-power-strong-prod}, we have
\begin{equation*}
\mathcal{G}_E(G)= \mathcal{G}_E(P_1) \boxtimes \cdots \boxtimes \mathcal{G}_E(P_r) \boxtimes \mathcal{G}_E(P_{r+1}) .
\end{equation*}
We can now proceed in an identical manner to the proof of 
\cref{Prop: VC enh Nilpo no cyclic no Q}. We only need to use \cref{dom of gen Q_2^n in enhced} to ensure that any separating set of $\mathcal{G}_E(Q_{2^k})$ has cardinality atleast $2$ and the size of a minimum connected component of $\mathcal{G}_E^{**}(Q_{2^k})$ is $2$. Thus, we can prove  
$\kappa(\mathcal{G}_E(G))=\mu$ 
where $\mu$ is taken from \eqref{Equa: mu for nilpo grp having Q}.
The remaining part of the proof follows from the observation that if $G= G_1 \times Q_{2^k} \times \mathbb{Z}_n$, then $\kappa(\mathcal{G}_E(G))= n \kappa (\mathcal{G}_E(G_1 \times Q_{2^k}))$. 
\end{proof}

\

\subsection{Abelian groups}
Here, we will explicitly calculate the vertex connectivity of the enhanced power graph of any finite abelian group. 
From \cref{Prop: VC enh Nilpo no cyclic no Q}, we can see that it is sufficient to investigate the minimum possible size of a connected component in $\mathcal{G}_E^{*}(P)$ where $P$ is a finite abelian $p$-group, with the form $$P= \mathbb{Z}_{p^{t_1}} \times  \cdots \mathbb{Z}_{p^{t_k}} $$ where $1 \leq t_1 \leq \dots \leq t_k.$ Therefore, in order to prove \cref{thm:vc-abelian-exact}, we first prove the following lemmas that tell us about the number of roots of an element in $P$ and the minimum possible size of a connected component in $\mathcal{G}_E^{*}(P).$

For a $p$-group $P$ and two elements $x,y \in P$, we call $y$ to be a \emph{$m$-th root} of $x$ if $m$ is the smallest positive integer such that $y^m=x$.

\begin{lemma}
\label{lem:number-of-roots-p-order-element}
Let $P= \mathbb{Z}_{p^{t_1}} \times  \cdots \times \mathbb{Z}_{p^{t_k}},$ where $1 \leq t_1 \leq \dots \leq t_k$ and $z\in P$ be an element of order $p$. Then, $z$ has at least  $(p-1)\frac{p^{t_1k}-1}{p^{k}-1} $ many roots. 
\end{lemma}
\begin{proof}
We can write $z=(z_1,z_2, \dots, z_k),$ where \[z_i \in \{0, p^{t_i-1}, 2p^{t_i-1}, \dots, (p-1)p^{t_i-1} \} \text{ for } 1 \leq i \leq k.\] The number of elements $y$ such that $y$ is $a$-th root of $z$, 
where $\gcd(a,p)=1$, is $p-1$. Let $y=(y_1,y_2,\dots, y_k)$ be a $p$-th root of $z$. For this to happen, each $y_i$ has $p$ choices. Moreover, each such choice gives a $p$-th root of $z$. Thus, the number of $p$-th roots of $z$ is $p^k$.
Therefore, the number of $ap$-th roots of $z$, where $\gcd(a,p)=1$, is $(p-1)p^k.$ Let $1 \leq r \leq t_1-1$ and we count the number of $ap^{r}$-th roots of $z$.  

Let $y=(y_1,y_2,\dots, y_k)$ be a $p^r$-th root of $z$. For this to happen, each $y_i$ has $p^r$ choices and each such choice gives a $p^r$-th root of $z$.  Thus, the number of $p^r$-th roots of $z$ is $p^{kr}$. Therefore, the number of $ap^r$-th roots of $z$, where $\gcd(a,p)=1$, is $(p-1)p^{rk}.$ Now $z$ may or may not have $p^r$-th roots. Summing over, we can see that $z$ has at least \[(p-1)+(p-1)p^k+(p-1)p^{2k}+\dots+(p-1)p^{(t_1-1)k}= (p-1)\frac{p^{t_1k}-1}{p^{k}-1} \] 
many roots. This completes the proof. 
\end{proof} 
\begin{lemma}\label{Lemma:min-pos-size-abel}
Let $P= \mathbb{Z}_{p^{t_1}} \times \mathbb{Z}_{p^{t_2}} \times \dots \times \mathbb{Z}_{p^{t_k}},$ where $1 \leq t_1 \leq \dots \leq t_k.$ Then the minimum possible size of a connected component in $\mathcal{G}_E^{*}(P)$ is exactly $(p-1)\frac{p^{t_1k}-1}{p^{k}-1} .$   
\end{lemma}
\begin{proof}
We consider the element $$x= (p^{t_1-1}, p^{t_k-1}, \dots, p^{t_r-1}).$$
Let $C_x$ be the component of $\mathcal{G}_E^{*}(P)$ containing $x$. We now count the cardinality of $C_x$. Every element in $C_x$ is adjacent to $x$. Moreover, as $x$ is of order $p$, every element is a root of $x$. The number of elements $y$ such that $y$ is $a$-th root of $x$, where $\gcd(a,p)=1$, is $p-1$. Let $y=(y_1,y_2,\dots, y_k)$ be a $p$-th root of $x$. For this to happen, each $y_i$ has $p$ choices. Thus, the number of $p$-th roots of $x$ is $p^k$. Therefore, the number of $ap$-th roots of $x$, where $\gcd(a,p)=1$, is $(p-1)p^k.$ Let $y=(y_1,y_2,\dots, y_k)$ be a $p^2$-th root of $x$. For this to happen, each $y_i$ has $p^2$ choices. Thus, the number of $p^2$-th roots of $x$ is $p^{2k}$. Therefore, the number of $ap^2$-th roots of $x$, where $\gcd(a,p)=1$, is $(p-1)p^{2k}.$ We can continue this till $t_1-1$. Let $y=(y_1,y_2,\dots, y_k)$ be a $p^{t_1-1}$-th root of $x$. For this to happen, each $y_i$ has $p^{t_1-1}$ many choices. Thus, the number of $p^{t_1-1}$-th roots of $x$ is $p^{(t_1-1)k}$. Therefore, the number of $ap^{t_1-1}$-th roots of $x$, where $\gcd(a,p)=1$, is $(p-1)p^{(t_1-1)k}.$ Finally, we observe that the element $x$ cannot have a $p^{t_1}$-th root. 

Therefore, the total number of roots of $x$, and hence the size of the component $C_x$ is \[(p-1)+(p-1)p^k+(p-1)p^{2k}+\dots+(p-1)p^{(t_1-1)k}= (p-1)\frac{p^{t_1k}-1}{p^{k}-1}.\] By \cref{lem:number-of-roots-p-order-element}, any component $C_x$ of $\mathcal{G}_E^{*}(P)$ has size at least $(p-1)\frac{p^{t_1k}-1}{p^{k}-1}.$ This completes the proof.
\end{proof}
\begin{proof}[Proof of \cref{thm:vc-abelian-exact}]
The proof follows from \cref{Prop: VC enh Nilpo no cyclic no Q} and \cref{Lemma:min-pos-size-abel}.  
\end{proof}

\subsection*{Acknowledgement} 
 The second author acknowledges SERB National Post Doctoral Fellowship (File No. PDF/2021/001899) and NBHM Post Doctoral Fellowship during the preparation of this work and profusely thanks Science and Engineering Research Board, Govt. of India for this funding.  The first and second authors
acknowledge the DST FIST program - 2021.
 The second author also acknowledges excellent working conditions in the Department of Mathematics, Indian Institute of Science.

\bibliographystyle{amsplain}
\bibliography{gen-inv-lcp.bib}

\end{document}